\documentclass[12pt,a4paper, reqno]{amsart}
\usepackage{ucs} 
\usepackage[T1]{fontenc}
\usepackage[utf8]{inputenc}
\usepackage[english]{babel}
\usepackage{amsmath}
\usepackage{amsfonts}
\usepackage{amssymb}
\usepackage{hyperref}
\usepackage{amsxtra,latexsym,amscd,amsthm,marvosym,multirow}
\usepackage[a4paper,inner=2.3cm, outer=2.3cm, top=2.3cm, bottom=2.3cm]{geometry}
\usepackage{setspace}
\usepackage[pdftex]{graphicx}
\usepackage{framed, fancybox}
\usepackage{epstopdf}
\usepackage{enumerate}
\usepackage{csquotes}

\usepackage{tikz}
\usetikzlibrary{decorations.pathreplacing}
\usepackage{xifthen,tikz}
\usetikzlibrary{lindenmayersystems}
\usetikzlibrary[shadings]
\usetikzlibrary{calc,intersections,through,backgrounds}
\usepackage{morefloats}
\usepackage{subcaption}
\usepackage{color}
\usetikzlibrary{calc,patterns,angles,quotes}
\newtheorem{theorem}{Theorem}[section]
\newtheorem{lemma}{Lemma}[section]
\newtheorem{proposition}[theorem]{Proposition}

\theoremstyle{remark}
\newtheorem{remark}{Remark}[section]

\newcommand{\h}{\hspace*{.24in}}

\def\beq{\begin{eqnarray*}}\def\eeq{\end{eqnarray*}}
\def\bq{\begin{equation}}\def\eq{\end{equation}}

\newcommand{\abs}[1]{\lvert#1\rvert}

\newcommand{\RR}{\mathbb{R}}

\newcommand{\ZZ}{\mathbb{Z}}

\newcommand{\dd}{\textup{d}}
\newcommand{\norm}[1]{\left\|#1\right\|}

\makeatletter

\@addtoreset{equation}{section}  
\makeatother

\begin{document}
\title[Boundary effects on the magnetic Hamiltonian dynamics]{Boundary
  effects on the magnetic Hamiltonian dynamics in two dimensions}
\author[D. T. Nguyen, N. Raymond, S. V\~u Ng\d{o}c]{Th\d{o} Nguyen
  Duc, Nicolas Raymond, San V\~u Ng\d{o}c}

\maketitle

\begin{abstract}
  We study the Hamiltonian dynamics of a charged particle submitted to
  a pure magnetic field in a two-dimensional domain. We provide
  conditions on the magnetic field in a neighbourhood of the boundary
  to ensure the confinement of the particle. We also prove a formula
  for the scattering angle in the case of radial magnetic fields.
\end{abstract}



\section{Introduction}
\subsection{Magnetic Hamiltonian dynamics}
This article is concerned with the dynamics of a charged particle in a
smooth bounded domain $\Omega\subset\RR^2$ in the presence of a non
homogeneous magnetic field $\textbf{B}$. The motion of a particle of
charge $e$ and mass $m$ under the action of the Lorentz force can be
expressed by Newton's equation
\begin{equation}\label{eq.Newton}
m\ddot{q}= e \dot{q}\times \textbf{B}\,,
\end{equation}
where $q =(q_1, q_2,q_3)^T \in \mathbb{R}^3$. To simplify our
discussion, we assume that $e=1$ and $m=1$. The vector field
$\textbf{B}$, defined on $\Omega$, is assumed to be smooth and to
satisfy the Maxwell equation $\nabla\cdot\textbf{B}=0$. For our target
problem in two dimensions, we suppose that $\textbf{B}$ is
perpendicular to the plane $\RR^2$, \textit{i.e.},
$\textbf{B}(q)=(0,0,b(q))$. This assumption forces particles lying in
the $\RR^2$ plane and whose initial velocities are in the plane to
stay in this same plane for all time. Since a vector field in
$\mathbb{R}^3$ can be identified with a 2-form, we write the magnetic
field as $\textbf{B} = b(q) \dd q_1\wedge \dd q_2$. Then, if there is
a 1-form $\textbf{A} =A_1 \dd q_1 + A_2 \dd q_2$ such that
$\dd \textbf{A} = \textbf{B}$, we can write~\eqref{eq.Newton} in
Hamiltonian form.  Consider, for all
$(q,p)\in \mathbb{R}^2 \times \mathbb{R}^2$,
\begin{equation}\label{Ham in standard}
\mathcal{H}(q,p) = \frac{\|p-\textbf{A}(q)\|^2}{2} \h\,,
\end{equation}
where $\|.\|$ denotes the Euclidean norm on $\mathbb{R}^2$.

The matrix representing the right cross product with $\mathbf{B}$ in
the canonical basis is
\[M_{\mathbf{B}}= J^T_{\textbf{A}}-J_{\textbf{A}}\,,\]
where $J_{\mathbf{A}}$ is the Jacobian matrix of $\mathbf{A}$. Hence
Newton's equation~\eqref{eq.Newton} becomes
\[\ddot{q}=M_{\mathbf{B}}\dot{q}\,,\]
so that
\begin{equation*}
\frac{d}{dt}\left( \dot{q}+\textbf{A}(q) \right) = J_{\textbf{A}}^T \dot{q}\,.
\end{equation*} 
By introducing the momentum variable $p = \dot{q}+\textbf{A}(q) $, we
see that $\mathcal{H}(q,p)=\frac{1}{2}\norm{\dot q}^2$ is the kinetic
energy of the system, and $(q,p)$ evolves according to the Hamiltonian
flow associated with $\mathcal{H}$:
\begin{equation}\label{eq.Hamilonian}
\left\{\begin{split}
\dot{q}&=\partial_{p}\mathcal{H}(q,p)\\
\dot{p}&=-\partial_{q}\mathcal{H}(q,p)
\end{split}\right..
\end{equation}
We shall always assume that $q\mapsto b(q)$ is locally
Lipschitz-continuous, ensuring that the system~\eqref{eq.Hamilonian}
has a unique local maximal solution, thanks to the Cauchy-Lipschitz
theorem. Then, the vector potential $\mathbf{A}$ will always be chosen
to be $C^1$-smooth.

\subsection{Two questions}

From now on, we call $b$ the magnetic field and it is identified with the $2$-form 
\[b(q_1,q_2) \dd q_1 \wedge \dd q_2=\dd \left(A_{1}\dd q_{1}+A_{2}\dd q_{2}\right)\,.\]
This article addresses two classical dynamical problems: confinement and scattering.
\begin{enumerate}[-]
\item (Confinement) Consider a charged particle in the magnetized region $\Omega$. A natural question is the following:

\centerline{\enquote{Will the particle reach the boundary in finite time?}}
\noindent We will provide a precise answer to this question, depending
on the behaviour of the magnetic field at the boundary and on the
initial conditions. Our results will improve recent results by Martins
in \cite{GM17}. In particular, we will see that, even if the magnetic
field is infinite at the boundary, some trajectories can escape from
$\Omega$. This kind of (open) problems is mentioned in \cite[Section
1.4]{VT11}.
\item (Scattering) Consider a charged particle outside the magnetized
  region $\Omega$. Before it reaches the region $\Omega$, the
  trajectory is a straight line. If it enters the region $\Omega$,
  does the particle escape from it in finite time? And, if it does so,
  what is the deviation angle between the ingoing and outgoing
  directions? We will explicitly answer these questions in the case of
  radial magnetic fields and when $\Omega$ is a disc. In this case,
  the angular momentum commutes with the Hamiltonian and allows a
  reduction to a one degree of freedom system.
\end{enumerate}
For both problems, we provide numerical illustrations of our results.

These questions have intrinsic physical motivations. Their answers
allow a better understanding of the classical dynamics of charged
particles in magnetic fields.  The description of the classical
trajectories has also many applications, for instance, at the quantum
level. The quantum aspect of the trapped trajectories can be related
to the essentially self-adjoint character of the magnetic Laplacian
(see \cite{VT11, GI09, GI11, RS75}). It is also a key point to
describe the spectrum/resonances of magnetic Laplacians. As far as the
authors know, whereas the description of the magnetic dynamics has
allowed to estimate the spectrum of magnetic Laplacians (see
\cite{RVN15, HKRVN16}), no result seems to exist to estimate their
resonances near the real axis. Investigating the trapped trajectories
is a necessary step in this direction.

In the regime of large magnetic field and small energy, a special
treatment of the confinement problem can be done and takes advantage
of the near-integrable structure of the Hamiltonian dynamics, either
via Birkhoff normal form~\cite{RVN15}, or KAM
theorems~\cite{castilho01}. On the contrary, our results here will
give more explicit initial conditions and allow regimes where the
guiding center motion is not necessarily meaningful.

\subsection{Organization of the article}
The article is organized as follows. In Section \ref{sec.main-results}, we state our main results about confinement and scattering. Section \ref{sec.proofs} is devoted to the proofs.


\section{Statements}\label{sec.main-results}

\subsection{Confinement problem}

\subsubsection{Tubular coordinates}
In order to state our results, it is convenient to introduce tubular coordinates near the boundary of $\Omega$, following the analysis of \cite{GM17}.

We assume that the connected components of $\partial\Omega$ are
$C^2$-smooth closed curves without self-intersections. Let
$\mathcal{C}$ be a connected component of $\partial \Omega$. It can be
parametrized by its arc length
$\gamma : \mathbb{R}/L\mathbb{Z} \to \mathcal{C}$ where $L$ is the
length of $\mathcal{C}$.

There exists $\delta>0$ such that
\begin{equation}\label{eq.psi}
\psi :\left\{ 
\begin{split}
  (0,\delta)\times \mathbb{R}/L\mathbb{Z} &\to\Omega_{\mathcal{C}}(\delta)\\
  (n,s)\qquad &\mapsto\gamma(s)+nN(s)=q
\end{split}\right.
\end{equation}
is a smooth diffeomorphism. $N(s)$ denotes the inward pointing normal at $\gamma(s)$ and
\[\Omega_{\mathcal{C}}(\delta)=\{q\in\Omega : \dd \left(x,\mathcal{C}\right)<\delta\}\,.\]
Note that
\begin{equation}\label{eq.Bns}
\textbf{B} = b(q) \dd q_1\wedge \dd q_2 = b(\psi(n,s))(1-n\kappa(s)) \dd s\wedge \dd n\,,
\end{equation}
where $\kappa(s)$ is the signed curvature of $\mathcal{C}$ at $\gamma(s)$.
In these coordinates, we can write
\begin{equation*}
\textbf{A} = A_n(n,s) \dd n + A_s(n,s) \dd s
\end{equation*}
with $A_n , A_s$ defined on $(0,\delta)\times \mathbb{R}/L\mathbb{Z}$ such that
\begin{equation}
  \frac{\partial A_s}{\partial n} - \frac{\partial A_n}{\partial s} =:
  B(n,s)=- b(\psi(n,s))(1-n\kappa(s))\,.
\end{equation} 
Via the tubular coordinates, we can define the symplectic change of coordinates
\begin{equation}\label{eq.Psi}
\Psi :\left\{ 
\begin{split}
(0,\delta)\times\mathbb{R}/L\mathbb{Z}\times\mathbb{R}^2&\to\Omega_{\mathcal{C}}(\delta)\times\mathbb{R}^2\\
(n,s, p_{n}, p_{s})&\mapsto (\psi(n,s),((\dd  \psi)_{(n,s)}^{-1})^\mathrm{T}(p_{n}, p_{s}))=(q,p)
\end{split}\right.,
\end{equation}
where we have explicitly $p=(1-n\kappa(s))^{-1}p_{s}\gamma'(s)+p_{n}N(s)$.

The Hamiltonian takes the form (see Lemma \ref{lem.cartesian.to.normal}):
\begin{equation}\label{eq.Htubular}
H(n,s,p_n,p_s) = \frac{1}{2}(p_n - A_n(n,s) )^2 + \frac{(p_s-A_s(n,s))^2}{2(1-\kappa(s)n)^2}\,.
\end{equation}

\subsubsection{General confinement theorems}
We can now state our confinement results.
Our first theorem provides a sufficient condition on $\mathbf{B}$ so that no trajectory can escape from $\Omega$.
\begin{theorem}\label{theo.confinement0}
For every connected component $\mathcal{C}$ of $\partial \Omega$, we assume that
\begin{equation} \label{eq.lim.of.integral}
\lim_{n \to 0} \left\lvert \int_{n}^{\delta_{\mathcal{C}}}\int_{0}^{L_{\mathcal{C}}} B(\eta,\xi)\dd \xi \dd \eta\right\rvert  =+ \infty\,,
\end{equation}
and that there exists $M_{\mathcal{C}}\geq 0$ such that, for all $(n,s) \in (0,\delta_{\mathcal{C}})\times \mathbb{R}/L_{\mathcal{C}}\mathbb{Z}$,
\begin{equation}\label{boundedness}
\left| B(n,s)-\frac{1}{L_\mathcal{C}}\int_{0}^{L_\mathcal{C}} B(n,\xi) \dd \xi \right| \leq M_{\mathcal{C}}\,.
\end{equation}
Then the magnetic Hamiltonian dynamics is complete (i.e. no solution
of \eqref{eq.Hamilonian}, starting in $\Omega$, reaches
$\partial\Omega$ in finite time).
\end{theorem}

Of course, given a starting point $q\in\Omega$, only the components
$\mathcal{C}$ that bound the connected component of $q$ in $\Omega$
need to be taken into account. Actually, there is a more quantitative
version of the previous theorem.
\begin{theorem}\label{theo.confinement1}
Consider a connected component $\mathcal{C}$ of $\partial\Omega$. Let 
\[K = \sup_{s \in \mathbb{R}/L\mathbb{Z}} \vert \kappa(s)\vert\,,\qquad K' = \sup_{s \in\mathbb{R}/L\mathbb{Z}} \vert\kappa'(s) \vert\,.\]
We assume that, for some $\epsilon\in(0,1)$, $\delta$ satisfies $0<\delta \leq \epsilon/K$. We assume that there exists $M\geq 0$ such that, for all $(n,s) \in (0,\delta)\times\mathbb{R}/L\mathbb{Z}$,
\begin{equation}\label{boundness}
\left| B(n,s)-\frac{1}{L}\int_{0}^{L} B(n,\xi) \dd \xi \right| \leq M\,.
\end{equation}
Consider $T>0$ and $q(t) = \psi(n(t),s(t)) $ a trajectory contained in
$\Omega_{\mathcal{C}}(\delta)$ for $t\in [0,T]$ with energy $H_0$. Let
\begin{equation}\label{eq.fn}
f(n) = -\frac{1}{L}\int_{n}^{\delta}\int_{0}^{L} B(\eta,\xi)\dd \xi \dd \eta
\end{equation}
and assume that
\begin{equation} \label{eq:f very large}
\liminf_{n \to 0}  |f(n)|  > C(T)\,,
\end{equation}
where
\begin{multline*}
C(T) = \left| \dot{s}(0)[1-\kappa(s(0))n(0)] + \int_{n(0)}^{\delta}\int_{0}^{L} B(\eta,\xi)\dd \xi \dd \eta \right|\\
+\sqrt{2H_0}(1+\epsilon) +\left( M\sqrt{2H_0} + \frac{2H_0 K' N}{1-\epsilon} \right)  T\,.
\end{multline*}
Let $g$\footnote{such a function g always exists.} be a continuous and strictly decreasing function  such that 
\[\lim_{n\to0} g(n)=\liminf_{n\to0}|f(n)|\,,\qquad g\leq |f|\,\quad \mathrm{ on }\,\, [0,\delta]\,.\] 
Then, $g$ takes the value $C(T)$ and, for all $t\in[0,T)$,
\begin{equation}
n(t)> g^{-1}(C(T))\,.
\end{equation}
\end{theorem}

\begin{remark}
  Theorems \ref{theo.confinement0} and \ref{theo.confinement1} are
  improvements of \cite[Theorems 1\&2]{GM17}. They tell us that a
  particle in $\Omega$ never reaches the boundary of $\Omega$. In
  \cite{GM17}, it is assumed that $\partial_{s}B$ is integrable:
\begin{equation}\label{eq.intMartins}
\sup_{s\in\mathcal{C}}\int_{0}^N |\partial_{s}B(m,s)|\dd m<+\infty\,,
\end{equation}
and the question of removing this
assumption was explicitly mentioned as important (\emph{op. cit.,
  section 3}). Our theorems give a partially positive answer to this question,
thus allowing for magnetic fields having wilder tangential behaviors.
\begin{enumerate}[-]
\item Theorem \ref{theo.confinement0} generalizes \cite[Theorem
  1]{GM17} by replacing the integrability assumption by
  \eqref{boundedness}. This allows in particular to consider a
  magnetic field (on the unit disc) of the form
  \[B(n,s)=\frac{1}{n}+\sin\left(\frac{\chi(s)}{n}\right)\,,\]
  where $\chi$ is a smooth function supported in $(-\pi,\pi)$ such
  that $\chi'(0)\neq 0$ and $\chi(0)=0$. For this magnetic field, it
  is easy to check that \eqref{eq.intMartins} is not satisfied.  In
  fact, the $C^\infty$ smoothness is actually not required; in order to
  draw Figure \ref{fig:Pic1}, we took, for simplicity, a small
  perturbation of $\chi(s)=\arcsin(\sin(s))$.

\begin{figure}[h] 
\centering
\includegraphics[width=0.5\textwidth ,height = 6 cm]{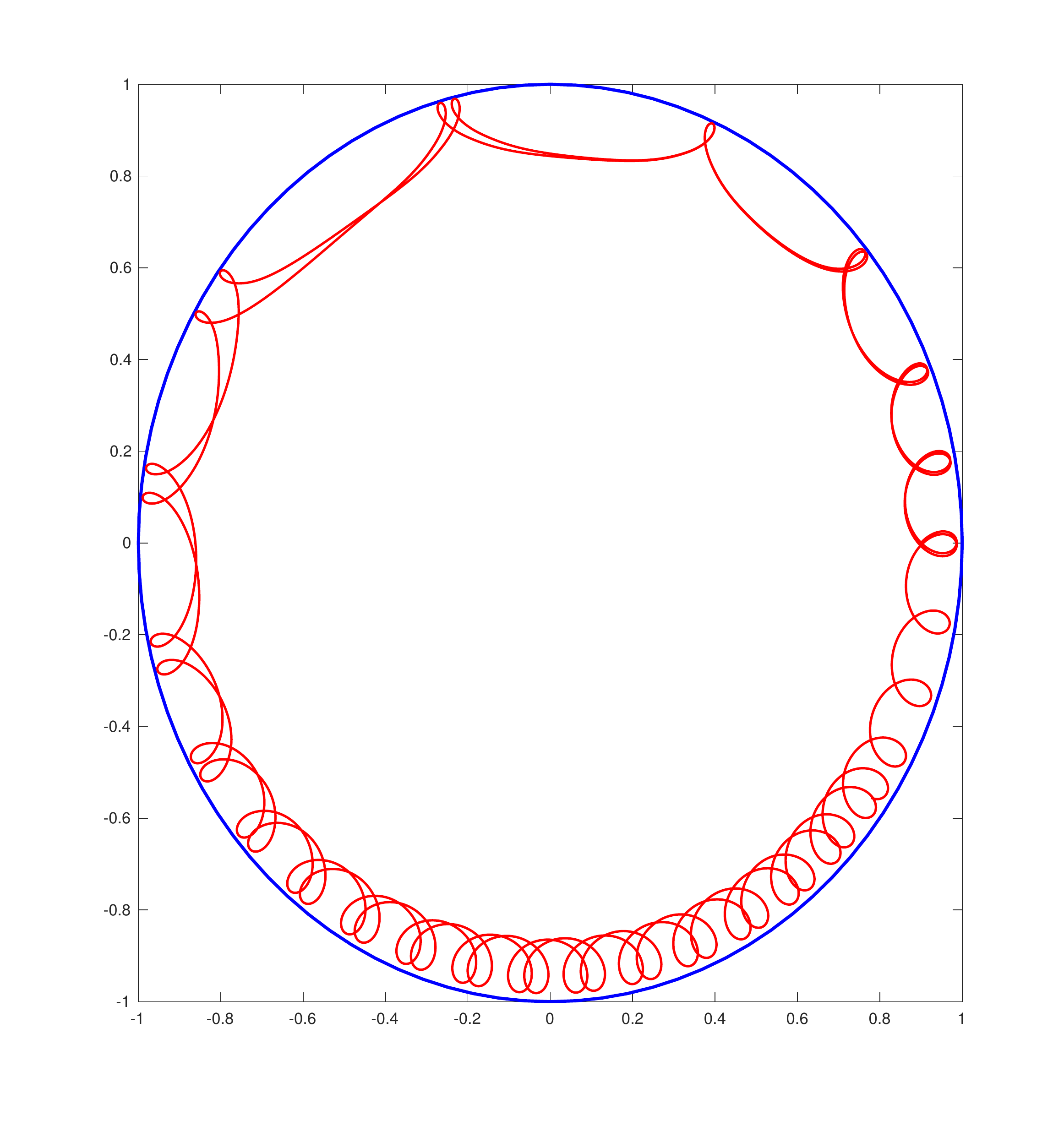}
\caption{A trajectory obtained with a magnetic field on the unit disc
  that is strong near the boundary with a non-integrable tangential derivative:\\
  $\displaystyle B(q) =\frac{1}{1-\sqrt{q_{1}^2+q_{2}^2}} + \sin\left(\frac{\arcsin(q_2)}{1-\sqrt{q_{1}^2+q_{2}^2}} \right) + 5q_1^3 -7q_2 \,.  $}
   \label{fig:Pic1}
\end{figure}

\item An explicit lower bound for the escaping time of a magnetized region is given in \cite[Theorem 2]{GM17} in the case when
\begin{equation}\label{eq.Mh}
B(n,s) = \frac{M}{n^\alpha} + h(n,s)\,, \qquad \alpha \geq 1\,.
\end{equation}
where $M \neq 0$ and $h$ is bounded and smooth in
$\Omega_{\mathcal{C}}(\delta)$, and so that \eqref{eq.intMartins}
holds. Theorem \ref{theo.confinement1} implies \cite[Theorem 2]{GM17},
and also provides an explicit lower bound for magnetic fields that are
not in the form \eqref{eq.Mh}, see Figure \ref{fig:Con7} where the
magnetic field changes sign infinitely many times.
\end{enumerate}

\end{remark}
\begin{figure}[h] 
\centering
\includegraphics[width=0.5\textwidth, height = 6 cm]{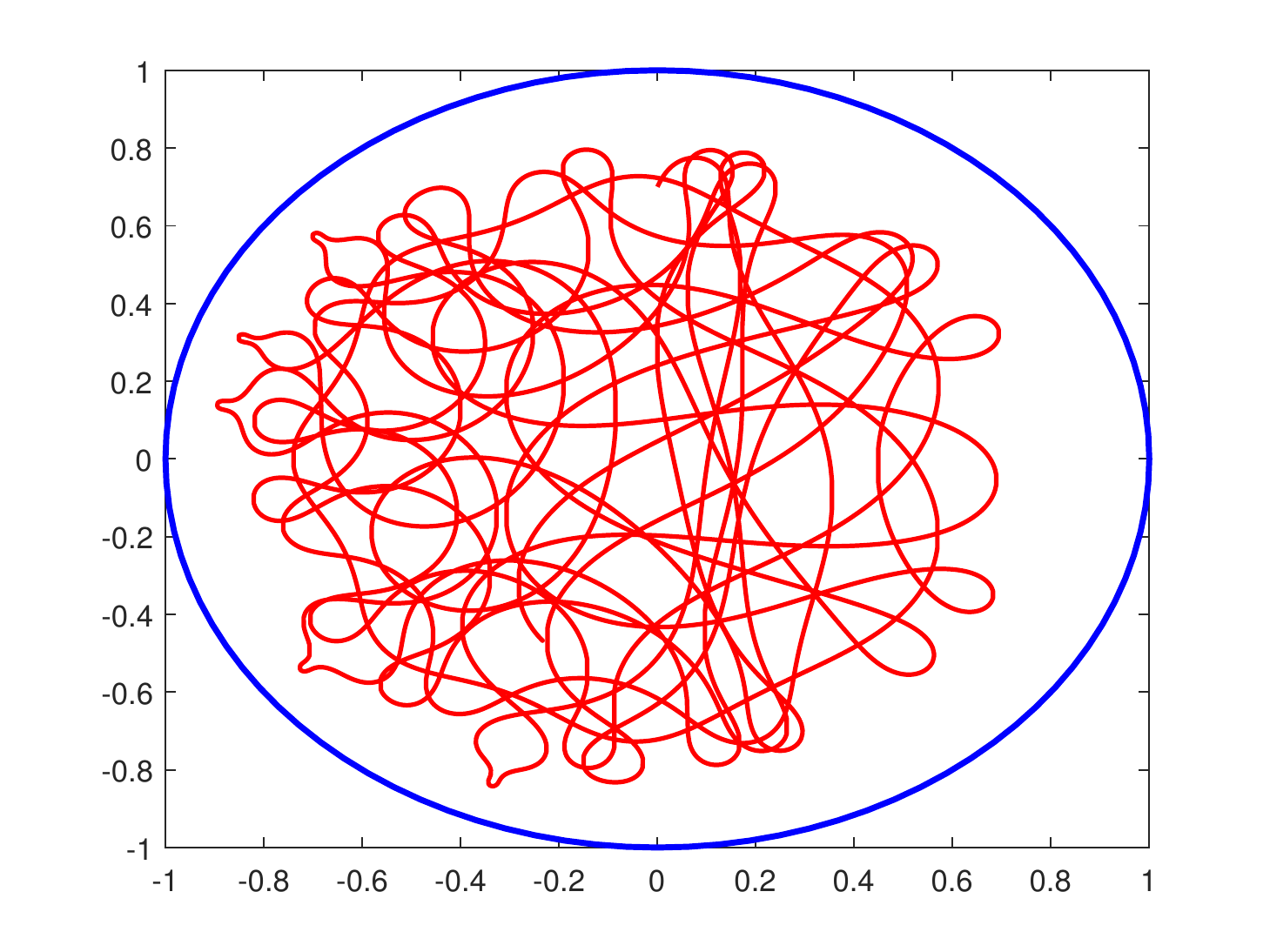}
\caption{A trajectory obtained with a magnetic field on the unit disc
  that strongly oscillates near the boundary:\\
  $\displaystyle B(q) = \frac{\frac{1}{2}-\sin\left(
      \frac{1}{1-\sqrt{q_1^2+q_2^2}} \right)
  }{(1-\sqrt{q_1^2+q_2^2})^2} +10q_1 -2q_1^2-10q_2^2\,.  $
  } \label{fig:Con7}
\end{figure}

\subsubsection{Confinement results in the radial case}
When $\Omega=D(0,1)$ and when $B$ is radial, the dynamics is
completely integrable, and hence can be entirely described by a one
degree of freedom Hamilonian; concerning the confinement problem, this
of course leads to stronger results.
\begin{proposition}\label{prop.of.confinement2}
  Let $q(t)=(q_1(t),q_2(t))$ be a solution to \eqref{eq.Hamilonian}
  starting at $t=0$ from inside the unit disc. If the initial data
  $(q(0), \dot{q}(0))$ satisfies either~\textbf{H1} or \textbf{H2}
  below:
\begin{enumerate}
\item[\textbf{H1}:]
\begin{equation}\label{eq.criterion.1}
\liminf_{r\to 1^-}\left|\frac{1}{2\pi}\int_{\|q(0)\|\leq \|q\|\leq r} B(q)\dd q-\det(q(0),\dot{q}(0))\right|>\|\dot{q}(0)\|\,,
\end{equation}
\item[\textbf{H2}:]
\begin{equation}\label{criterion2}
\liminf_{r\to 1^-}\left|\frac{1}{2\pi}\int_{\|q(0)\|\leq \|q\|\leq r} B(q)\dd q-\det(q(0),\dot{q}(0))\right|=\|\dot{q}(0)\|\,,
\end{equation}
and
\begin{equation}\label{criterion2'}
\limsup_{r\to 1^-}\frac{\left|\frac{1}{2\pi}\int_{\|q(0)\|\leq \|q\|\leq r} B(q)\dd q-\det(q(0),\dot{q}(0)) \right|-\Vert \dot{q}(0) \Vert}{r-1} <0\,,
\end{equation}
\end{enumerate}
then the solution exists for all $t\geq 0$, and there exists $\eta \in[0,1)$ such that
\begin{equation}
\forall t\geq 0\,,\quad\Vert q(t) \Vert < \eta\,.
\end{equation} 
\end{proposition}
One can find situations where none of the hypothesis of Proposition
\ref{prop.of.confinement2} hold and the trajectory can be arbitrarily
close to the boundary. 
(see Figure \ref{fig.escape-turning}: this unusual behavior can be
explained by a critical point of the radial Hamiltonian at $r=1$, see
\eqref{eq.red.H}).
\begin{figure}[h] 
\centering
\includegraphics[width=0.4\textwidth]{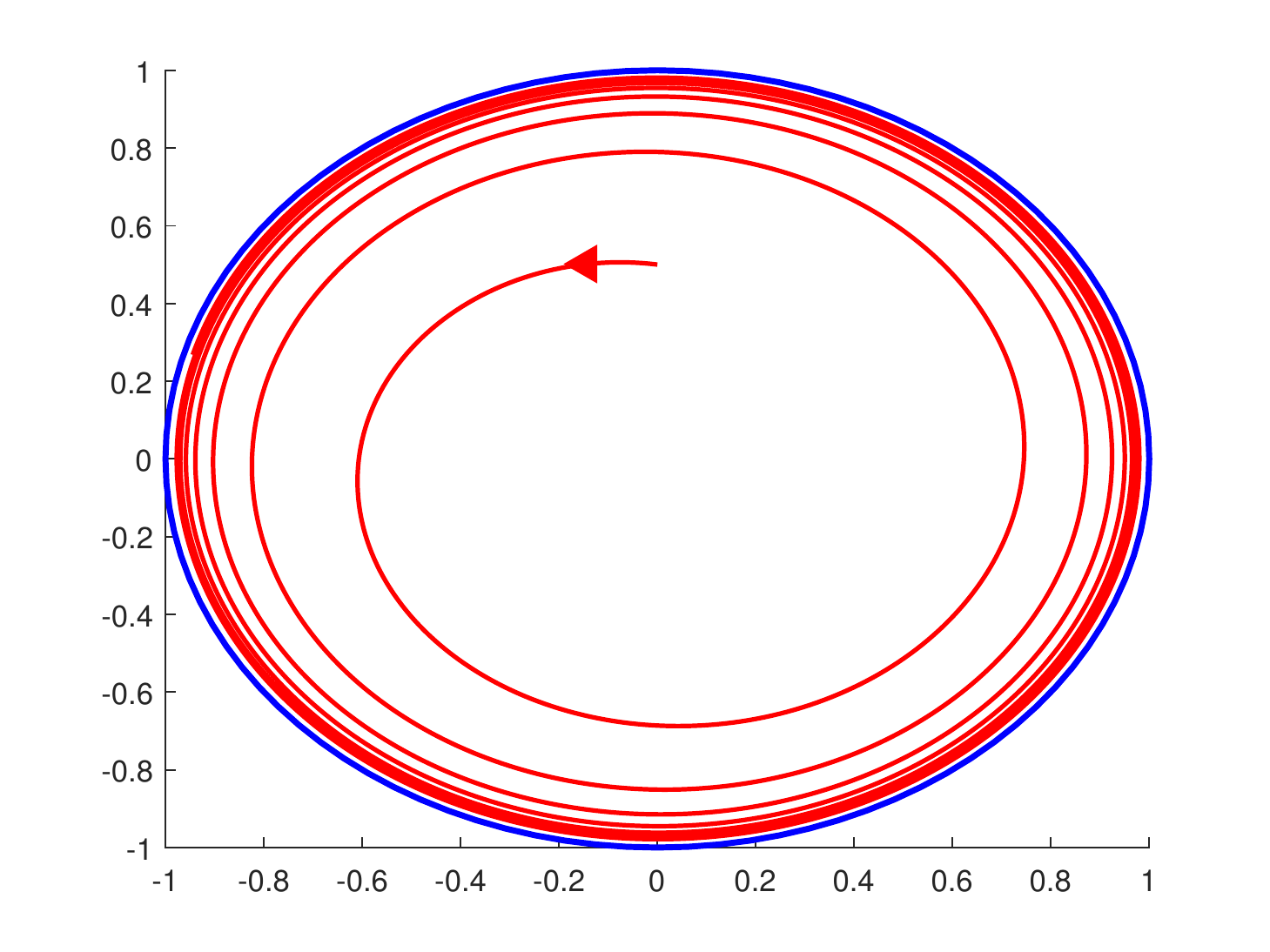}
\caption{$B(r)= e^{-r}-\frac{2}{r} $ .} \label{fig.escape-turning}
\end{figure}

If the magnetic field is $L^1$-integrable near the boundary of
$\Omega$, we can prove that there exist trajectories escaping from
$\Omega$ in finite time.  In particular, even if the magnetic field is
infinite at the boundary, the confinement is not ensured.
\begin{proposition}\label{prop.confinement.3}
When
\begin{equation}
\label{equ:L1}
\limsup_{r\to 1^-} \left\lvert \int_{D(0,r)} B(q)\dd q \right\rvert < +\infty\,,
\end{equation} 
there exists a trajectory starting in $\Omega$ and reaching the boundary in finite time.
\end{proposition}

Of course, even under assumption~\eqref{equ:L1}, some trajectory may
be confined, depending on initial conditions (see Figure \ref{fig:confine} where the simulations
are performed with $B(r)=\ln^2(1-r)$).

\begin{figure}[] 
\centering
\begin{subfigure}[t]{0.5\textwidth}
\centering\includegraphics[width=0.7\textwidth]{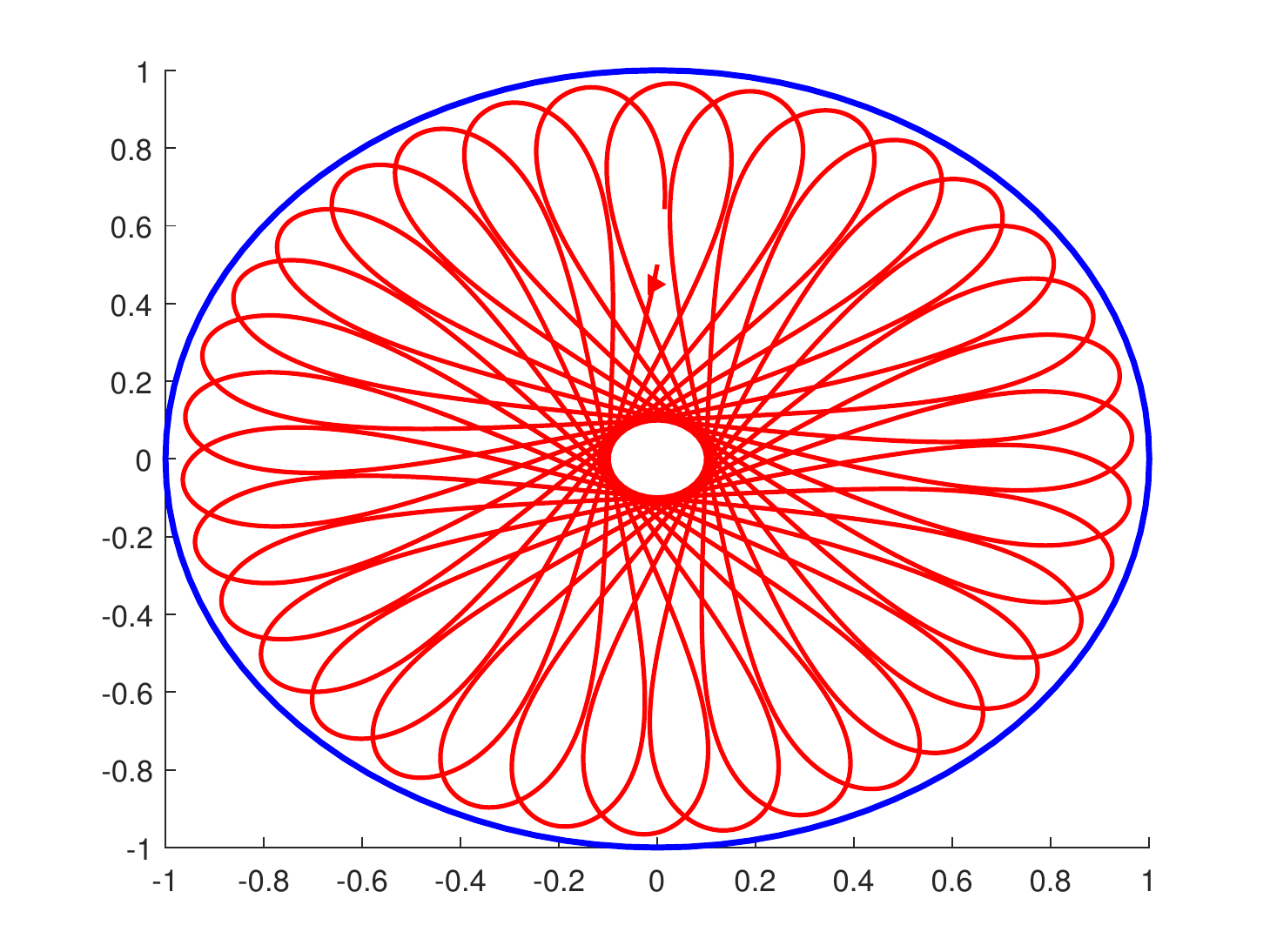}
\end{subfigure}
~
\begin{subfigure}[t]{0.5\textwidth}
\centering\includegraphics[width=0.7\textwidth]{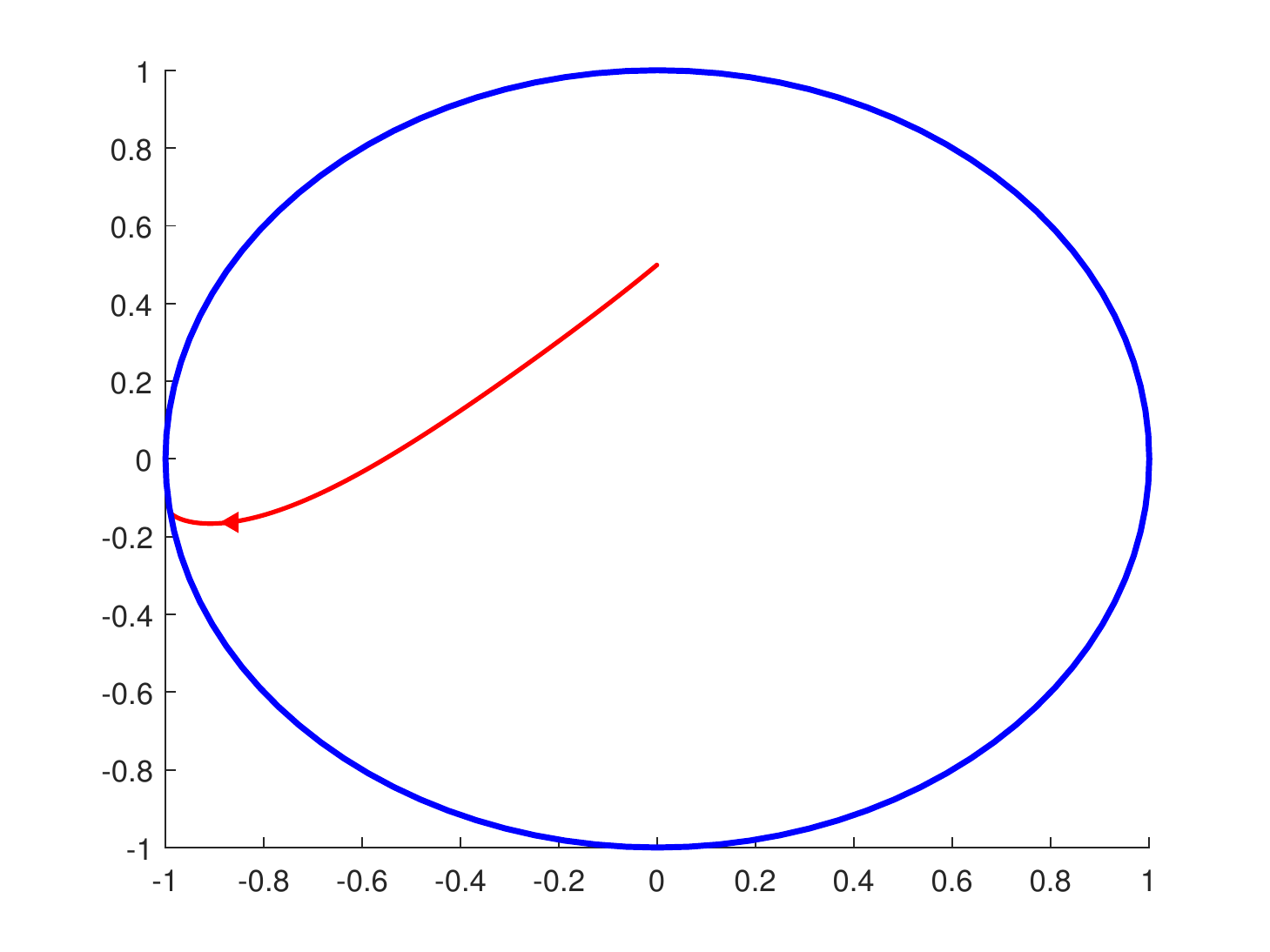}
\end{subfigure}
\caption{$B(r)=\ln^2(1-r) $: the particle is confined or not.}\label{fig:confine}
\end{figure}

\subsection{Scattering in the radial case}
Let us now describe our scattering result in the radial case. We
assume that $\mathbf{B}_{|\Omega}$ admits a locally Lipschitz
extension in a neighbourhood of $\Omega$.


In polar coordinates, we have
\[
\textbf{B} = B(r)r \dd r \wedge \dd \theta=\dd \left(G(r) \dd
  \theta\right)\,,
\]
where
\begin{equation*}
G(r) = \int_0^r \tau B(\tau)  \dd \tau\,.
\end{equation*}
Via the symplectic change of coordinates
\[
\begin{split}
  \mathbb{R}_{+}^*\times \RR/2\pi\ZZ\times\mathbb{R}^2&\to (D\setminus\{0\})\times\mathbb{R}^2\\
  (r,\theta,p_{r},p_{\theta})&\mapsto \left(r\cos\theta,r\sin\theta,
    \cos \theta p_r - \frac{\sin \theta}{r}p_\theta,\sin \theta p_r
    +\frac{\cos \theta}{r}p_\theta\right)=(q,p)
\end{split}\quad\,,
\]
the Hamiltonian becomes
\begin{equation}\label{eq.Ham.in.polar}
  \tilde H(r,\theta,p_r,p_\theta) = \frac{p_r^2}{2}+ \frac{(p_\theta - G(r))^2}{2r^2}\,,
\end{equation}
In particular, the angular momentum $p_{\theta}$ is constant along the
flow and we consider the reduced one dimensional Hamiltonian on
$T^*\RR^*_+$
\begin{equation}\label{eq.red.H}
  H(r,p_r) := \frac{p_r^2}{2}+V(r)\,,\qquad V(r) :=  \frac{(p_\theta - G(r))^2}{2r^2}\,,
\end{equation}
where $V\in C^1(\RR^*_+)$.  We notice that (see, for example, Lemma
\ref{lem.cartesian.to.normal})
\[v_{r}=p_{r}\,,\qquad v_{\theta}=r^{-1}(p_{\theta}-G(r))\,,\]
where $v_{r}$ and $v_{\theta}$ are the classical radial and tangential components of the velocity $v$.

We consider a charged particle with energy $H_0$ arriving into the
disk with velocity $v_{1}$. In particular,
$H_{0}=\frac{1}{2}\|v_{1}\|^2$. If the particle escapes from the disc
with velocity $v_{2}$ (see Figure \ref{Scattering}), we have
$\norm{v_2}=\norm{v_1}$, and a natural question is to compute the
(scattering) angle between these two vectors. Let
$\omega \in (-\pi,\pi]$ be the oriented angle between $v_1$ and $v_2$.

\begin{theorem}\label{theo.scattering}
  Consider a trajectory starting on $\partial\Omega$, with velocity
  $v_{1}\neq 0$ and entering $\Omega$. This means that either
  $v_{r}<0$, or $v_{r}=0$ and $\frac{B(1)}{v_{\theta}}<-1$. We define
  $\gamma$ the angle between the outward pointing normal and $v_{1}$.

We also assume
\begin{enumerate}[\rm i.]
\item either that the equation $V(r)=H_{0}$ has a solution for $r\in(0,1)$ and that the closest solution to $1$, denoted by $r^*$, satisfies $V'(r^*)<0$.
\item or, only when $p_{\theta}=0$, that the equation $V(r)=H_{0}$ has no solution.
\end{enumerate}

Then the trajectory escapes from $\Omega$ in finite time with velocity $v_{2}$, and we can compute the scattering angle $\omega \mod 2\pi$:
\begin{enumerate}[\rm i.]
\item either the trajectory does not pass through the origin and
\[\omega=\alpha+\pi-2\gamma\,,\]
where
\begin{equation}\label{alpha angle}
\alpha = 2 \int_{r^*}^{1}  \frac{p_\theta-G(r)}{r\sqrt{2H_0 r^2- (p_\theta-G(r))^2} } \dd r\,,
\end{equation}
\item or the trajectory passes through the origin (in this case $p_{\theta}=0$) and
\[\omega=\alpha-2\gamma\,,\]
where
\begin{equation}\label{alpha angle}
\alpha = 2 \int_{0}^{1}  \frac{-G(r)}{r\sqrt{2H_0 r^2- G(r)^2} } \dd r\,.
\end{equation}
\end{enumerate}

\end{theorem}

\begin{figure}[!h]
 \centering
\begin{tikzpicture}
\coordinate (O) at (0,0);
\coordinate (A) at (1,0);
\coordinate (B) at (-4,3);
\coordinate (C) at (-4,-3);
\coordinate (D) at (-5,3.75);
\coordinate (E) at (-3.2,2);
\coordinate (F) at (-4,-3);

 \draw[red,thick] (0,0) circle (5);
 
 \draw[red,thick,dashed] (0,0) circle (2);
 
 \draw[color=black,very thick, ->] (-4,3) -- (-3.2,2);
 
 \draw[color=black, very thick, ->] (-4,-3) -- (-4.8,-4);
 
 \draw [thick,dashed] plot[smooth, tension = 0.4] coordinates{ (-4,3)  (-2,0) (-4,-3) };
 
 \draw[color=gray, thick, -] (-5,3.75)--(0,0);
 \draw[color=gray, thick, -] (-5,-3.75)--(0,0);
 
 \draw[color=gray, thick, -] (0,0)--(5,0);
 
 \node[scale=1] at (-3.7,2) {$\vec{v}_1$};
 \node[scale=1] at (-4,-4) {$\vec{v}_2$};
 \node[scale=1] at (1,-0.5) {$r^*$};
 
\pic [draw, ->,angle radius = 0.3 cm, "$\theta_1$", angle eccentricity=2] {angle = A--O--B};
\pic [draw, ->,angle radius = 1 cm, "$\theta_2$", angle eccentricity=1.4] {angle = A--O--C};
\pic [draw, -,thick, angle radius = 0.3 cm, "$\gamma$", angle eccentricity=2] {angle = D--B--E};
\pic [draw, ->,thick, angle radius = 0.5 cm, "$\alpha$", angle eccentricity=1.4] {angle = B--O--F};
 
\draw [decorate,decoration={brace,amplitude=10pt},xshift=0pt,yshift=0pt] (2,0) -- (0,0) node [black,midway,xshift=0.5 cm] {};
\end{tikzpicture}
\caption{ The scattering arrows. }
\label{Scattering}
\end{figure}
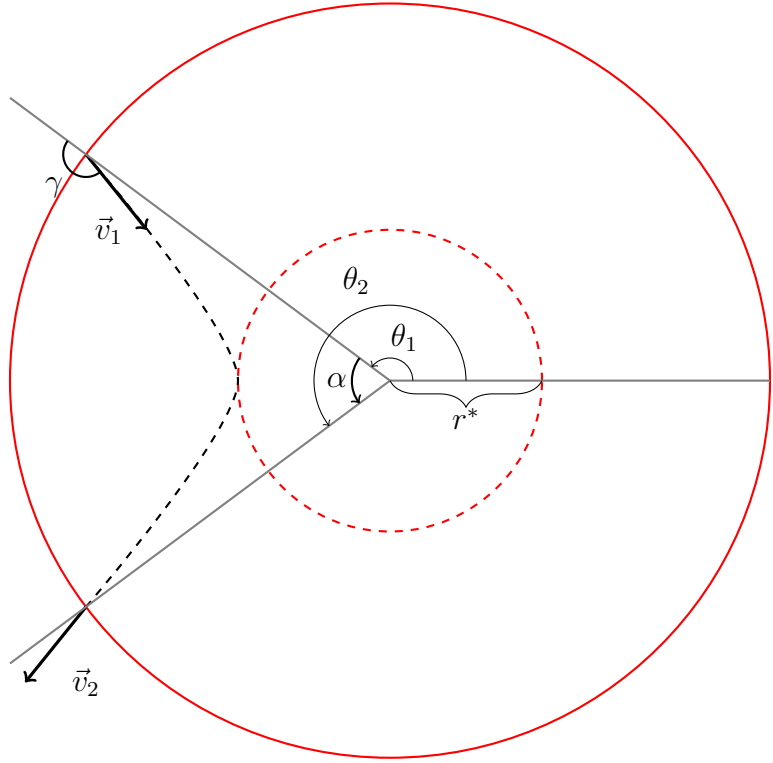


\section{Proofs}\label{sec.proofs}
\subsection{Proof of Theorems \ref{theo.confinement0} and \ref{theo.confinement1}}


To reach the boundary, the particle has to be close to a connected
component $\mathcal{C}$ of $\partial \Omega$. Thus, we can assume
that, for all $t\in[0,T)$,
\begin{equation*}
q(t)\in\Omega_{\mathcal{C}}(\delta)\,.
\end{equation*} 
Modifying the vector potential corresponds to a symplectic
transformation of the form $(q,p)\mapsto (q, p+\dd S(q))$, for some
smooth function $S$, and hence does not modify the trajectory of the
particle. Thus, we consider the function
\begin{equation*}
  \alpha(n,s) =   \frac{s}{L} \int_{0}^{L}  B(n,\xi) \dd \xi 
  -\int_0^s B(n,\xi) \dd \xi\,.
\end{equation*}
Notice that $\alpha(n,\cdot)$ is $L$-periodic. Recalling \eqref{eq.fn}
and letting $\textbf{A} = \alpha(n,s) \dd n + f(n) \dd s$, we have
$ \textbf{B}=\dd\textbf{A}$.

By \eqref{eq.Htubular}, the corresponding Hamiltonian is
\begin{equation*}
  H(n,s,p_n,p_s) = \frac{(p_n - \alpha(n,s) )^2}{2} + \frac{(p_s-f(n))^2}{2(1-\kappa(s)n)^2}\,.
\end{equation*}
Concerning Hamilton's equations, we have in particular
\begin{equation*}
\dot{n} = p_n- \alpha(n,s)\,,\qquad \dot{p}_s =\tilde{B}(n,s)\dot{n} - \frac{(p_s-f(n))^2}{(1-\kappa(s)n)^3} \kappa'(s)n\,,
\end{equation*}
where 
\begin{equation*}
  \tilde{B}(n,s) = \frac{1}{L}\int_{0}^{L}B(n,\xi) \dd \xi-B(n,s)\,.
\end{equation*}
We recall that, for all $t\in[0,T)$, $H(n(t),s(t),p_n(t),p_s(t)) = H_0$. We get
\begin{equation}\label{estimates 1}
\begin{split}
\vert \dot{n} \vert &\leq \sqrt{2H_0}\\
\abs{p_s-f(n)} &\leq \sqrt{2 H_0}(1+\epsilon)\\
\left\lvert \frac{(p_s-f(n))^2}{(1-\kappa(s)n)^3} \kappa'(s)n \right\rvert &\leq \frac{2H_0 K' \delta}{1-\epsilon}  \,,
\end{split}
\end{equation}
where in the last estimates we have used the notation of
Theorem~\ref{theo.confinement1} and in particular
$\abs{\kappa} n \leq K \delta \leq \epsilon$.
With our assumption~\eqref{boundness} on $\tilde{B}(n,s)$, we find,
for all $t\in[0,T)$,
\begin{eqnarray*}
\left\lvert p_s(t)\right\rvert  \leq   \vert p_s(0)\vert + \left( M\sqrt{2H_0} + \frac{2H_0 K' \delta}{1-\epsilon} \right) T\,,
\end{eqnarray*}
and thus
\begin{eqnarray}\label{eq.f <C}
\left\lvert  f(n(t)) \right\rvert &\leq& \left\lvert p_s(t)  \right\rvert + \left\lvert p_s(t) - f(n(t)) \right\rvert \leq C(T)\,,
\end{eqnarray}
with
\begin{equation*}
  C(T) = \vert p_s(0)\vert +\sqrt{2H_0}(1+\epsilon) +\left( M\sqrt{2H_0} + \frac{2H_0 K' \delta}{1-\epsilon} \right)  T\,.
\end{equation*}
If the trajectory reaches the boundary at $t=T$ , then
\begin{equation*}
\lim_{t\to T} n(t) = 0\,.
\end{equation*}
This, with \eqref{eq.f <C} and \eqref{eq.lim.of.integral}, gives a contradiction. This proves Theorem \ref{theo.confinement0}.

Now, consider a function $g$ as in Theorem \ref{theo.confinement1}. We have, for all $t\in[0,T)$,
\begin{equation*}
g(n(t)) \leq |f(n(t))| \leq  C(T)\,.
\end{equation*}
From~\eqref{eq:f very large}, we have $\lim_{n\to 0} g(n)>C(T)$; hence
$g$ must take the value $C(T)$ and the conclusion follows.

\subsection{Proof of Proposition \ref{prop.of.confinement2}}
Let us recall \eqref{eq.red.H}. The assumptions of Proposition \ref{prop.of.confinement2} can be written in terms of $V$. 
\begin{enumerate}[\rm (H1)]
\item If
\begin{equation}\label{eq.liminfV1>H}
\liminf_{r\to 1^-} V(r) > H_0\,,
\end{equation}
we consider $\eta=\sup\{x\in(0,1) : V(x)=H_{0}\}\in(0,1)$. Consider a
trajectory $(q(t), p(t))$ with $q(0)\in D(0,1)$. We can assume that
$q(0)\neq 0$.  Let $T$ be the maximal time of existence in
$D(0,1)$. By energy conservation, we have, for all $t\in[0,T)$,
\[V(r(t))\leq H_{0}\,,\]
so that $r(t)\leq\eta$. 

Note that \eqref{eq.liminfV1>H} means 
\[\liminf_{r\to 1^-} \left\lvert G(r) - p_\theta \right\rvert>   \sqrt{2 H_0}\,.\]
Using the usual complex coordinate in the plane $\RR^2$, we can write
$\dot{q}=\left(\dot{r}+i\dot{\theta}r\right)e^{i\theta}$ and thus
\begin{equation*}
  \det (q(t), \dot{q}(t)) = r^2(t) \dot{\theta}(t) = p_\theta - G(r(t))\,.
\end{equation*} 
Finally, we notice that $\|\dot{q}(0)\|=\sqrt{2H_{0}}$ and write
\[G(r) - p_\theta = G(r)-G(r(0)) - [p_\theta - G(r(0))]\,,\] 
which gives~\eqref{eq.criterion.1}.

\item If
\begin{equation}\label{eq.liminfV1=H}
\liminf_{r\to 1^-} V(r) = H_0\,,
\end{equation}
and
\begin{eqnarray*}
\limsup_{r\to 1^-}\frac{V(r)- H_0}{r-1} <0\,,
\end{eqnarray*}
then we must again have 
\[
\sup\{x\in(0,1) : V(x)=H_{0}\} < 1\,,
\]
and we can proceed as above.
\end{enumerate}
\subsection{Proof of Proposition \ref{prop.confinement.3}}
Consider $p_{\theta}=0$. Let
$\abs{V}_\infty := \sup_{r\in(0,1)}\abs{V(r)}$. By assumption,
$\abs{V}_\infty< +\infty$.

Let $r(0) \in (0,1)$ and choose $p_r(0)>0$ such that $ p_r^2(0) = 2\left(|V|_{\infty}-V(r(0))\right)+v^2$, with $v>0$. Since, for all $t\in[0,T)$,
\begin{equation*}
\frac{p_r^2(t)}{2}+V(r(t))=\frac{p_r^2(0)}{2}+V(r(0))\,,
\end{equation*}
we get $\dot{r}(t)= p_r(t) > v$ so that
\begin{equation*}
r(t) > vt+r(0)\,.
\end{equation*}
The particle escapes at $t=\frac{1-r(0)}{v}$.
\subsection{Proof of Theorem \ref{theo.scattering}}

We distinguish between the cases $p_{\theta}=0$ and $p_{\theta}\neq 0$.

\subsubsection{Case when $p_{\theta}\neq 0$}
In this case, $\lim_{r\to 0} V(r)=+\infty$; hence, due to energy
conservation, the trajectory does not approach the origin.
\begin{enumerate}[\rm i.]
\item Assume that $p_{r}(0)<0$. We have $V(1)<H_{0}$ and we can
  consider the right most turning point $r^*\in(0,1)$. By definition
  $V(r^*)=H_{0}$, and necessarily $V'(r^*)\leq 0$.

If $V'(r^*)<0$, it is easy to check that $r$ reaches $r^*$ in finite time, say $t=t^*$. This time is given by
\[t^*=\int_{r^*}^{1}\frac{\dd r}{\sqrt{2(H_{0}-V(r))}}\,.\]
By symmetry, the escape time is $2t^*$. Since $\dot{\theta}=\frac{p_{\theta}-G(r)}{r^2}$, we have
\[\theta(t^*)-\theta(0)=\int_{0}^{t^*} \frac{p_\theta-G(r)}{r^2} \dd t=\int_{0}^{t^*} \frac{(p_\theta-G(r)) \dot{r}}{r^2 p_r} \dd t=\int_{0}^{t^*}  - \frac{(p_\theta-G(r)) \dot{r}}{r^2\sqrt{2(H_0-V(r))}  }\dd t\,,\]
so that
\[\theta(t^*)-\theta(0)= \int_{r^*}^{1}  \frac{p_\theta-G(r)}{r^2\sqrt{2(H_0-V(r))} } \dd r\,.\]
By symmetry, we have
\[\theta(2t^*)-\theta(0)=2\int_{r^*}^{1}  \frac{p_\theta-G(r)}{r^2\sqrt{2(H_0-V(r))} } \dd r\,.\]

If $V'(r^*)=0$, $(r^*,0)$ is a critical point of the Hamiltonian and we get that $r$ reaches $r^*$ in infinite time (see Figure \ref{confine4}).

\item Assume that $p_{r}(0)=0$. Then $V(1)=H_{0}$. By assumption (the
  trajectory enters $D(0,1)$), we have $V'(1)\geq 0$, \textit{i.e.},
  $(p_{\theta}-G(1))B(1)+(p_{\theta}-G(1))^2\leq 0$ . If $V'(1)=0$,
  the particle sits at a fixed point of the Hamiltonian system, and
  hence $r(t)\equiv 1$ is constant. If $V'(1)>0$, it enters $D(0,1)$
  and the discussion is the same as previously.
\end{enumerate}
\begin{figure}[h] 
\centering
\includegraphics[width=0.5\textwidth]{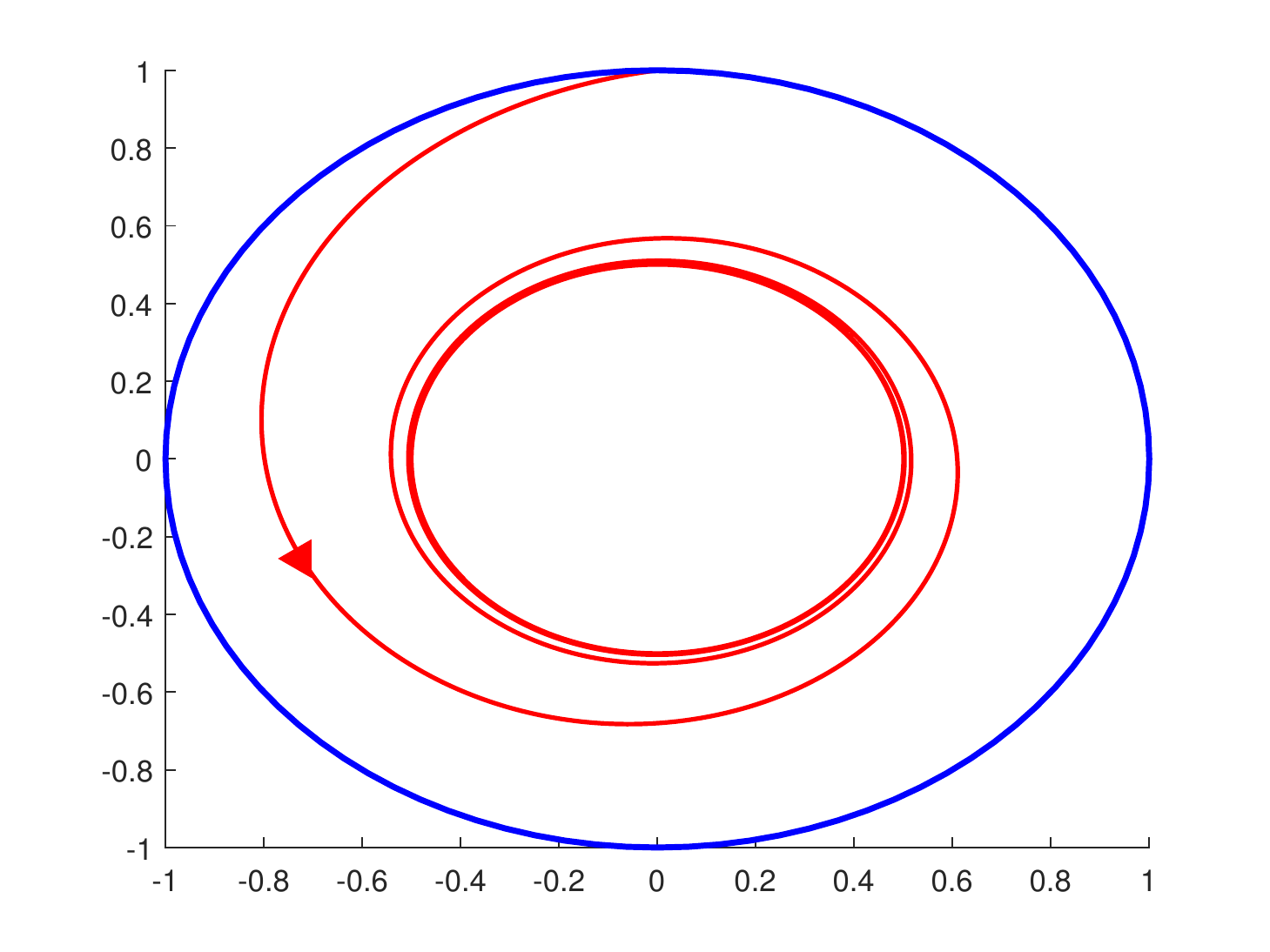}
\caption{$B(r)= e^{-r}-\frac{2}{r}$ .} \label{confine4}
\end{figure}

\subsubsection{Case when $p_{\theta}= 0$}
In this case, since $G(0)=0$, $V(r)=\frac{1}{2r^2}G(r)^2$ admits a
continuous extension at $r=0$.

\begin{enumerate}[\rm i.]
\item Assume that $p_{r}(0)<0$. We have $V(1)<H_{0}$. The existence of
  $r^*$ such that $V(r^*)=H_{0}$ is not ensured. If $V(r)<H_{0}$ on
  $[0,1]$, the particle reaches $r=0$ in finite time $t=t^*$:
\[t^*=\int_{0}^{1}\frac{\dd r}{\sqrt{2(H_{0}-V(r))}}\,.\]
We get, by symmetry,
\[\theta(2t^*)-\theta(0)=2\int_{0}^{1}  \frac{-G(r)}{r^2\sqrt{2(H_0-V(r))} } \dd r+\pi\,.\]
If there exists $r^*\in(0,1)$ such that $V(r^*)=H_{0}$, the trajectory does not reach the origin and the discussion is the same as in the case $p_{\theta}\neq 0$.

\item Assume that $p_{r}(0)=0$. The discussion is the same as when $p_{\theta}\neq 0$.

\end{enumerate}

\subsubsection{Scattering angle}
We can now end the proof of Theorem \ref{theo.scattering}. In terms of complex numbers, we can write
\[v_{1}=(v_{r}(0)+iv_{\theta}(0))e^{i\theta_{1}}\,,\quad v_{2}=(-v_{r}(0)+iv_{\theta}(0))e^{i\theta_{2}}\,.\]
The scattering angle is 
\[\theta_{2}-\theta_{1}+\mathrm{Arg}\left(\frac{-v_{r}(0)+iv_{\theta}(0)}{v_{r}(0)+iv_{\theta}(0)}\right)\,.\]
If $\gamma$ denotes the argument of $v_{r}(0)+iv_{\theta}(0)$, the scattering angle is thus
\[\theta_{2}-\theta_{1}+\pi-2\gamma\,.\]


\appendix
\section{Tubular coordinates}
\begin{lemma} \label{lem.cartesian.to.normal}
We write $\mathbf{A}=A_{1}\dd q_{1}+ A_{2}\dd q_{2}$. With \eqref{eq.psi}, we have
\[\mathbf{A}=A_{n}\dd n+A_{s}\dd s\,,\qquad \tilde A=(A_{n}, A_{s})^\mathrm{T}=(d\psi)^\mathrm{T}(A_{1}, A_{2})^\mathrm{T}\,.\]
We have
\begin{equation} \label{Ham in normal coor}
H(n,s,p_n,p_s) =\mathcal{H}\circ \Psi(n,s,p_{n}, p_{s})= \frac{(p_n - A_n(n,s))^2}{2} + \frac{(p_s- A_s(n,s))^2}{2(1-\kappa (s) n)^2}\,.
\end{equation}
Moreover, $v_{n}=p_n - A_n(n,s)$ and $v_{s}=(1-n\kappa(s))^{-1}(p_{s}-A_{s})$ are the normal and tangential component of $v$.
\end{lemma}
\begin{proof}
We write
\[2H(q,p)=\|p-A\|^2=\|(\dd\psi^{-1})^\mathrm{T}(\tilde p-\tilde A)\|^2=\langle (\dd\psi^{-1})(\dd\psi^{-1})^\mathrm{T}(\tilde p-\tilde A), \tilde p-\tilde A\rangle\,,\]
with $\tilde p=(p_{n}, p_{s})^{\mathrm{T}}$. Note that 
\begin{equation}\label{eq.dpsi-1T}
(\dd\psi^{-1})^\mathrm{T}=[N(s)\,\,, (1-n\kappa(s))\gamma'(s)]\,.
\end{equation}
We get
\[(\dd\psi^{-1})(\dd\psi^{-1})^\mathrm{T}=\begin{pmatrix}1&0\\ 
0&(1-n\kappa(s))^{-2}\end{pmatrix}\,.\]
Concerning the velocity $v$, we write
\[v=p-A=(\dd\psi^{-1})^\mathrm{T}(\tilde p-\tilde A)\,,\]
and we use \eqref{eq.dpsi-1T}.
\end{proof}


\end{document}